\documentclass[10pt,a4paper,twoside,reqno]{amsart}
\usepackage[utf8]{inputenc}
\usepackage[margin=1.5in]{geometry} 
\usepackage{amsthm}
\usepackage{amsmath}
\usepackage{amsfonts}
\usepackage{amssymb}
\usepackage{bbm}
\usepackage{url}
\usepackage{hyperref}
\usepackage[makeroom]{cancel}
\usepackage{enumerate}

\AtEndDocument{\bigskip{\footnotesize%
  \textsc{Aled Walker, Mathematical Institute, University of Oxford, Andrew Wiles Building, Radcliffe Observatory Quarter, Woodstock Road, Oxford, OX2 6GG.} \\ 
\texttt{walker@maths.ox.ac.uk}}}

\newtheorem{Theorem}{Theorem}

\newtheorem{Proposition}[Theorem]{Proposition}

\title{The primes are not metric Poissonian}
\author{Aled Walker}
\begin{document}
\begin{abstract}
It has been known since Vinogradov that, for irrational $\alpha$, the sequence of fractional parts $\{\alpha p\}$ is equidistributed in $\mathbb{R}/\mathbb{Z}$ as $p$ ranges over primes. There is a natural second-order equidistribution property, a pair correlation of such fractional parts, which has recently received renewed interest, in particular regarding its relation to additive combinatorics. In this paper we show that the primes do not enjoy this stronger equidistribution property. 
\end{abstract}
\maketitle

\section{Introduction}
Let $\mathcal{A}\subset \mathbb{N}$ be an infinite sequence of natural numbers, and let $A_N$ denote the first $N$ elements of $\mathcal{A}$. For $\alpha \in [0,1]$, we consider the sequence $\alpha \mathcal{A}$, taken modulo 1. Recall that the sequence $\alpha \mathcal{A}$ is \emph{equidistributed} in $\mathbb{R}/\mathbb{Z}$ if for every interval $I\subset \mathbb{R}/\mathbb{Z}$ one has 
\begin{equation}
\label{equidistribution definition}
\lim\limits_{N\rightarrow \infty}\frac{1}{N}\sum\limits_{x\in A_N}\mathbbm{1}_{I}(\alpha x) = \vert I\vert.
\end{equation} 
\noindent For many arithmetic sequences $\mathcal{A}$ of interest, the sequence $\alpha \mathcal{A}$ is equidistributed in $\mathbb{R}/\mathbb{Z}$ for all irrational $\alpha$. This is true for $\mathcal{A}=\mathbb{N}$ itself, or more generally the set of $k^{th}$ powers for any $k\in \mathbb{N}$, and, most pertinently for us, the set of primes. 

In this paper we will consider a strictly stronger notion of equidistribution. With notation as above, we define the pair correlation function \begin{equation}
\label{key definition of pair correlation function}
F(\mathcal{A},\alpha,s,N) :=\frac{1}{N} \sum\limits_{\substack{x_i,x_j\in A_N\\ x_i\neq x_j}} \mathbbm{1}_{[-s/N,s/N]}( \alpha(x_i - x_j)),
\end{equation}
\noindent where both the interval $[-s/N,s/N]$ and the sequence $\alpha \mathcal{A}$ are considered modulo $1$.\\

 Informally, $F(\mathcal{A},\alpha,s,N)$ counts the number of pairs $(\alpha x_i,\alpha x_j)$ such that the distance $\alpha x_i - \alpha x_j \text{ mod } 1$ is approximately $s$ times the average gap length of the sequence $\alpha A_N\text{ mod } 1$. Analysing the behaviour of $F(\mathcal{A},\alpha,s,N)$ for a specific $\alpha$ can require delicate Diophantine information about $\alpha$ (see \cite{H-B10}, \cite{RSZ01}), but one may instead settle for results which hold for almost all $\alpha$.
 
 In the setting of (\ref{equidistribution definition}), \emph{any} $\mathcal{A}$ satisfies the equidistribution property for almost all $\alpha$ (the sharpest results in this direction are due to Baker \cite{Ba81}). However, in the setting of pair correlations, the situation is more subtle. We say that the sequence $\mathcal{A}$ is `metric Poissonian' if for almost all $\alpha\in [0,1]$, and for all fixed $s>0$, we have 
\begin{equation}
\label{def of met poiss}
F(\mathcal{A},\alpha,s,N) = 2s(1+o_{\mathcal{A},\alpha,s}(1))
\end{equation} 
\noindent  as $N\rightarrow \infty$. Notice that if we had picked $N$ i.i.d. random variables $(X_n)_{n\in [N]}$ uniformly distributed on $\mathbb{R}/\mathbb{Z}$, instead of the sequence $\alpha A_N \text{ mod }1$, then as $N$ tends to infinity the equivalent pair correlation function would tend to $2s$ with high probability. Therefore (\ref{def of met poiss}) may be viewed as some strong indication that $\alpha \mathcal{A}\text{ mod }1$ exhibits the behaviour of a random sequence. The connection to the equidistribution property (\ref{equidistribution definition}) was recently made rigorous: indeed, three simultaneous papers \cite{ALP16,GL16, St17} recently showed that if (\ref{def of met poiss}) holds, for some fixed $\alpha$ and for all $s$, then for the same $\alpha$ one has that $\alpha \mathcal{A}$ is equidistributed in $\mathbb{R}/\mathbb{Z}$. \\

One might expect that, for the classical sequences $\mathcal{A}$ where $\alpha \mathcal{A}$ is equidistributed for irrational $\alpha$, one could prove that these sequences $\mathcal{A}$ are metric Poissonian. Indeed, for $k\geqslant 2$ and $\mathcal{A}$ the set of $k^{th}$ powers this was shown by Rudnick and Sarnak \cite{RS98}. However, the sequence $\mathcal{A}=\mathbb{N}$ is \emph{not} metric Poissonian. This follows from consideration of the continued fraction expansion of $\alpha$, but is in fact a special case of a more general phenomenon, connected to the large \emph{additive energy} of this particular set $\mathcal{A}$.

For a finite set $B\subset \mathbb{N}$ we define the additive energy $E(B)$ to be the number of quadruples $(b_1,b_2,b_3,b_4)\in B^4$ such that $b_1+b_2 = b_3+b_4$. If $\vert B\vert =N$, then we have the trivial bounds $N^2\ll E(B)\leqslant N^3$. For $x\in\mathbb{R}$, let us write $\Vert x\Vert$ for $\operatorname{min}_{y\in \mathbb{Z}}\vert x-y\vert$. Then, for an additive quadruple $(b_1,b_2,b_3,b_4)$ satisfying $b_1+b_2 = b_3+b_4$, obviously if $\Vert \alpha(b_1-b_3)\Vert \leqslant \frac{s}{N}$ then $\Vert \alpha(b_2-b_4)\Vert\leqslant \frac{s}{N}$. This is extremely different behaviour than that which would be seen if $\alpha b_1$, $\alpha b_2$, $\alpha b_3$, $\alpha b_4$ were genuinely i.i.d. uniform random variables on $\mathbb{R}/\mathbb{Z}$, and indeed we have the following result of Bourgain, which shows that all sets of nearly maximal energy fail to have the metric Poissonian property. 

\begin{Theorem}{\cite[Appendix]{ALLB16}}
\label{large energy theorem}
Suppose $$\operatorname{limsup}\limits_{N\rightarrow \infty} \frac{E(A_N)}{N^3}>0.$$ Then $\mathcal{A}$ is not metric Poissonian. 
\end{Theorem}
\noindent It is clear that the sequence $\mathcal{A}=\mathbb{N}$ satisfies the hypotheses of this theorem, and therefore this sequence is not metric Poissonian.

Remarkably, a near-converse to this theorem has also been proved to be true. 
 
 \begin{Theorem}
 \label{small energy theorem}
 Let $\delta>0$ be fixed, and suppose that $E(A_N) \ll_\delta N^{3-\delta}$ for this fixed $\delta$ and for every $N$. Then $\mathcal{A}$ is metric Poissonian. 
 \end{Theorem}

\noindent This theorem first appears as stated\footnote{The same result may be deduced from Theorem 3.2 of Harman's earlier book \cite{Ha98}, combined with the relevant modification of the variance estimate from page 69 of the same volume.} in the recent work of Aistleitner, Larcher and Lewko \cite{ALLB16}. It immediately implies the theorem of Rudnick-Sarnak on $k^{th}$ powers, and also earlier work on lacunary sequences \cite{RZ99}.  \\ 

It is natural to wonder whether there is a tight energy threshold for this problem. Although the truth seems unlikely to be so clean, it is certainly interesting to consider the behaviour of specific sets $\mathcal{A}$ which satisfy $N^{3-\varepsilon} \ll_\varepsilon E(A_N)\ll o(N^3)$ for all $\varepsilon>0$. In this paper, we prove the following result (answering a question posed by Nair\footnote{at the ELAZ 2016 conference in Strobl.}).
\begin{Theorem}[Main Theorem]
\label{Main Theorem}
The primes are not metric Poissonian.
\end{Theorem}

When $\mathcal{A}$ is the set of primes one has $E(A_N)\asymp N^3(\log N)^{-1}$, so certainly the primes are not included in the range of applicability of either Theorem \ref{large energy theorem} or Theorem \ref{small energy theorem}. 

In \cite{ALLB16}, Bourgain constructs a sequence $\mathcal{A}$ which is not metric Poissonian but nonetheless has $E(A_N) = o(N^3)$, thereby showing that the converse to Theorem \ref{large energy theorem} is false. A quantitative analysis of his argument shows that $E(A_N)\ll_{\varepsilon} N^3 (\log\log N)^{-\frac{1}{4} + \varepsilon}$ is achievable, for any $\varepsilon>0$. So, as an immediate corollary to Theorem \ref{Main Theorem}, we have an improved bound for the smallest energy $E(A_N)$ of the initial segments of a set $\mathcal{A}$ which is not metric Poissonian. \\\\\\

\noindent\textbf{Acknowledgements}

\noindent The author would like to thank Prof. R. Nair for making him aware of the central question of this paper, Prof. C. Aistleitner for comments on an earlier version, and Prof. B. J. Green for his doctoral supervision. A helpful conversation was also had with Sam Chow. The work was completed while the author was a Program Associate at the Mathematical Sciences Research Institute in Berkeley, which provided excellent working conditions. The author is supported by EPSRC grant no. EP/M50659X/1. 

\section{Proof of Theorem \ref{Main Theorem}}
The plan of the proof is as follows. For each fixed $\alpha$, we will try to find infinitely many $n$ such that $\Vert \alpha n\Vert $ is extremely small. Using such an $n$ we will be able to construct a scale $N$ and a small constant $s$ such that $\Vert \alpha mn\Vert\leqslant s/N$ for some initial segment of integers $m$. By a variant of a well-known result concerning the exceptional set for the Goldbach problem, we may show that many such $mn$ are represented many times as $p_i-p_j$ for two primes $p_i,p_j\leqslant p_N$, the $N^{th}$ prime. Combining all these observations will enable us to conclude, provided $s$ is small enough, that $F(\mathcal{P},\alpha,s,N)\geqslant c$ for some constant $c>2s$. Since this holds for infinitely many $N$, we cannot have $ F(\mathcal{P},\alpha,s,N) = 2s(1+o_{\alpha,s}(1))$ for all almost all $\alpha$ and for all $s>0$. In fact, we will show that, for almost all $\alpha$, this asymptotic \emph{fails} to hold.  \\

We now begin to consider the details of this argument. We will use the following result of Harman on Diophantine approximation (\cite[Theorem 4.2]{Ha98}, \cite{Ha88}). 

\begin{Theorem}
\label{Harman's theorem}
Let $\psi(n)$ be a non-increasing function with $0<\psi(n)\leqslant \frac{1}{2}$. Suppose that $$\sum\limits_n\psi(n)=\infty.$$ Let $\mathcal{B}$ be an infinite set of integers, and let $S(\mathcal{B},\alpha,N)$ denote the number of $n\leqslant N$, $n\in \mathcal{B}$, such that $\Vert n\alpha\Vert < \psi(n)$. Then for almost all $\alpha$ we have
\begin{equation}
S(\mathcal{B},\alpha,N) = 2\Psi(N,\mathcal{B}) + O_\varepsilon(\Psi(N)^{\frac{1}{2}}(\log \Psi(N)
)^{2+\varepsilon})
\end{equation}
\noindent for all $\varepsilon>0$, with implied constant uniform in $\alpha$, where $$\Psi(N) = \sum\limits_{n\leqslant N} \psi(n)$$ and $$\Psi(N,\mathcal{B}) = \sum\limits_{\substack{n\leqslant N\\n\in \mathcal{B}}} \psi(n).$$
\end{Theorem}
\noindent This theorem may be thought of as a flexible version of Khintchines's theorem on Diophantine approximation, in which one can further pass to approximations coming from a set $\mathcal{B}$, provided $\mathcal{B}$ is relatively dense. The quality of the error term in this theorem is much better than we need in our application, although it is important that there is no dependence on $N$ except through $\Psi(N)$. Earlier results of this type include a $\log N$ factor in the error, which would not have been adequate. \\

The other technical tool will be the standard bound on the size of the exceptional set in a Goldbach-like problem.

\begin{Theorem}
\label{binary Goldbach}
For a large quantity $X$, and natural number $n\leqslant X$, define $$r(n): = \sum\limits_{\substack{p_i,p_j\leqslant X\\ p_i-p_j = n}} \log p_i \log p_j.$$ Then for any $B>0$, and for all but $O_B(\frac{X}{\log ^B X})$ exceptional values of $n\leqslant X$, we have the approximation
\begin{equation}
\label{asymptotic formula for binary Goldbach}
r(n)= \mathfrak{S}(n)J(n) + O_B(\frac{X}{\log ^B X}),
\end{equation}
\noindent where 
\begin{equation*}
\mathfrak{S}(n) :=
\begin{cases}
2\prod\limits_{p\geqslant 3}\left(1-\frac{1}{(p-1)^2}\right)\prod\limits_{\substack{p\vert n\\p\geqslant 3}}\frac{p-1}{p-2} & n\text{ even, }\\
0 & n\text{ odd}
\end{cases} 
\end{equation*}
\noindent is the singular series, and $$J(n) = \int\limits_{-\infty}^{\infty}\left\vert\int\limits_{0}^{X} e(\beta u) \, du \right\vert^2 e(-\beta n) \, d\beta$$ is the singular integral.
\end{Theorem}
\begin{proof}
This result follows by trivial modifications of the usual argument for the binary Goldbach problem, originally due (independently) to van der Corput, Estermann, and \u{C}udakov. The clearest reference is \cite[Chapter 3.2]{Va97}, or, for a more modern approach, one may consider the proof of Theorem 19.1 in \cite[Chapter 19]{IwKo04}.
\end{proof}

We combine these two key ingredients in the following proposition. 
\begin{Proposition}
\label{key proposition}
There exists a small absolute $c>0$, such that for almost all $\alpha\in [0,1]$, and for all fixed $s>0$, there exist infinitely many $n$ satisfying:
\begin{enumerate}[(i)]
\item $\Vert \alpha n\Vert < \frac{s}{ n\log n}$\\
\item At least $c\log n$ of the numbers $n$, $2n$, $\cdots$, $\lfloor \frac{1}{10}\log n\rfloor n$ are expressible in at least $c\frac{n}{\log n}$ ways as the difference $p_1 - p_2$ of two primes $p_1,p_2 \leqslant \frac{1}{2} n\log n$. 
\end{enumerate}  
\end{Proposition}

\begin{proof}[Proof of Proposition \ref{key proposition}]
Let $c>0$ be a quantity to be specified later. With this $c$, let $\mathcal{B}$ be the set of natural numbers $n$ which satisfy (ii), and let $\psi(n) = \operatorname{min}(\frac{1}{2},\frac{s}{ n \log n})$. It is to this $\mathcal{B}$ and this $\psi$ that we will apply Theorem \ref{Harman's theorem}.\\

We claim that $\mathcal{B}$ is relatively dense. Indeed, let $K$ be a large integer, and let $n$ and $m$ be natural numbers restricted to the ranges $K\leqslant n<2K$ and $1\leqslant m\leqslant \lfloor \frac{1}{10}\log 2K\rfloor$. For notational convenience we let $X$ denote the quantity $\frac{1}{2} K\log K$, and we consider Theorem \ref{binary Goldbach} with this $X$. We say that the pair $(n,m)$ is \emph{exceptional} if $nm$ lies in the exceptional set from Theorem \ref{binary Goldbach} for which the asymptotic formula (\ref{asymptotic formula for binary Goldbach}) fails to hold. [Note that $nm\leqslant X$, so Theorem \ref{binary Goldbach} applies in this setting.] 

The map $(n,m)\mapsto nm$ is at most $\frac{1}{10}\log 2K$-to-$1$, due to the restricted range of $m$. Since the exceptional set from Theorem \ref{binary Goldbach} has size at most $O_B(\frac{K}{\log ^B K})$, for all $B>0$, we conclude that there are at most $O_B(\frac{K}{\log ^B K})$ exceptional pairs $(n,m)$, for all $B>0$. So certainly there are at least $(1-O_B(\log ^{-B} K))K$ values of $n\in [K,2K)$ such that the asymptotic formula for $r(nm)$ holds for all $m\leqslant \frac{1}{10} \log 2K$. Let $D_K$ denote this set of $n$. \\

$D_K$ is certainly very dense in $[K,2K)$, and we claim further that $D_K\subset \mathcal{B}$, provided we choose $c$ small enough. Combining the different scales $K$ will allow us to show that $\mathcal{B}$ is suitably dense. Indeed, let us analyse the asymptotic formula for $r(nm)$. When $nm$ is even, the singular series $\mathfrak{S}(nm)$ is always $\Omega(1)$. By Fourier inversion, the singular integral is exactly $(\mathbbm{1}_{[0,X]}\ast \mathbbm{1}_{[-X,0]})(nm)$, which is $\Omega(X)$ provided that $\vert nm\vert \leqslant \frac{1}{5} X$. But, by the choice of ranges for $n$ and $m$, this inequality is always satisfied. So, for all $n\in D_K$ and for all $m\leqslant \frac{1}{10}\log n$, such that $nm$ is even, we have $r(nm)\gg X$.

Removing the log weights on the primes, and recalling the definition of $X$, in particular we notice that there is some small absolute constant $c_1$ such the following holds: if $n\in D_K$ and if $nm$ is even, there are at least $c_1K/\log K$ pairs of primes $(p_i,p_j)$ with $p_i,p_j\leqslant \frac{1}{2}K\log K$ such that $p_i - p_j$ = $nm$. By the definition of $\mathcal{B}$, provided $c$ is chosen smaller than $\operatorname{min}(c_1,\frac{1}{100})$, we have that $D_K\subset \mathcal{B}$. \\

We may now prove that $\Psi(\mathcal{B},N)\gg_s \Psi(N)$ for large $N$. Indeed, 
\begin{align*}
\Psi(\mathcal{B},N)& = \sum\limits_{\substack{n\leqslant N\\ n\in \mathcal{B}}} \operatorname{min}(\frac{1}{2},\frac{s}{n\log n})\\
&\gg -O_s(1) + \sum\limits_{k=k_0}^{\lfloor\log_2 N\rfloor - 1 }\sum\limits_{\substack{2^k\leqslant n< 2^{k+1}\\n\in \mathcal{B}}} \frac{s}{2^k\log (2^k)}\\
&\gg -O_s(1) + \sum\limits_{k=k_0}^{\lfloor\log_2 N\rfloor -1}\sum\limits_{2^k\leqslant n< 2^{k+1}}(1-O_B(k^{-B})) \frac{s}{k 2^k}\\
&\gg -O_s(1) + \sum\limits_{k=1}^{\lfloor \log N \rfloor}\frac{s}{k}\\
&\gg_s \log\log N\\
&\gg_s \Psi(N).
\end{align*}

Therefore, applying Theorem \ref{Harman's theorem} to this set $\mathcal{B}$ and this function $\psi$, the main term from the conclusion of Theorem \ref{Harman's theorem} dominates the error term, and we conclude that for almost all $\alpha$ there are infinitely many $n\in \mathcal{B}$ satisfying $\Vert \alpha n\Vert < \frac{s}{n\log n}$. The proposition is proved. 
\end{proof}

With this moderately technical proposition proved, the deduction of Theorem \ref{Main Theorem} is extremely short. 
\begin{proof}[Proof of Theorem \ref{Main Theorem}]
Let $\Omega\subset [0,1]$ be the full-measure set of $\alpha$ for which Proposition \ref{key proposition} holds. Let $c$ be the constant from Proposition \ref{key proposition}, and fix some $s$ satisfying $0<2s<c^2$. Let $\alpha\in \Omega$, and fix a large $N$ to be one of the infinitely many natural numbers which satisfy the conclusions of Proposition \ref{key proposition}. 

By construction, we know that  $$\Vert \alpha N\Vert < \frac{s}{ N\log N}.$$ Therefore, for all $d\leqslant \frac{1}{10}\log N$, we have $$\Vert \alpha dN\Vert < \frac{s}{ N}.$$ But by the second conclusion of Proposition \ref{key proposition}, this implies that there are at least $c^2N$ pairs of distinct primes $p_i,p_j\leqslant \frac{1}{2}N\log N$ such that $$\Vert \alpha (p_i-p_j)\Vert < \frac{s}{ N}.$$ Since $P_N\sim N\log N$, and $N$ is large, this certainly implies that $$F(\mathcal{P},\alpha,s,N)\geqslant c^2 >2s.$$ This holds for infinitely many $N$, and therefore for all $\alpha \in \Omega$ we have $$F(\mathcal{P},\alpha,2s,N)\neq 2s(1+o(1))$$ as $N\rightarrow \infty$. Since $\Omega$ has measure 1, Theorem \ref{Main Theorem} is proved.
\end{proof}
\bibliographystyle{plain}
\bibliography{primesnotmetpoissonian}
\end{document}